\documentclass[12pt,letterpaper]{amsart}
\usepackage[utf8]{inputenc}
\usepackage{mathtools}
\usepackage{amssymb, amsmath, enumerate, amsthm, cite, colonequals, stmaryrd, lipsum, graphicx, fancyvrb,fancyhdr, indentfirst, times, setspace}
\usepackage[margin=1in]{geometry}
\usepackage{mathtools}
\makeatletter
\@namedef{subjclassname@2020}{\textup{2020} Mathematics Subject Classification}
\makeatother

\usepackage[colorlinks]{hyperref}
\hypersetup{
 colorlinks=true,
 linkcolor=blue,
 filecolor=magenta, 
 urlcolor=cyan,
}

\usepackage[capitalise]{cleveref}
\crefformat{equation}{(#2#1#3)}
\crefrangeformat{equation}{(#3#1#4--#5#2#6)}
\crefformat{enumi}{(#2#1#3)}
\crefrangeformat{enumi}{(#3#1#4--#5#2#6)}

\usepackage[pagewise]{lineno}
\let\oldequation\equation
\let\oldendequation\endequation
\renewenvironment{equation}{\linenomathNonumbers\oldequation}{\oldendequation\endlinenomath}
\expandafter\let\expandafter\oldequationstar\csname equation*\endcsname
\expandafter\let\expandafter\oldendequationstar\csname endequation*\endcsname
\renewenvironment{equation*}{\linenomathNonumbers\oldequationstar}{\oldendequationstar\endlinenomath}
\let\oldalign\align
\let\oldendalign\endalign
\renewenvironment{align}{\linenomathNonumbers\oldalign}{\oldendalign\endlinenomath}
\expandafter\let\expandafter\oldalignstar\csname align*\endcsname
\expandafter\let\expandafter\oldendalignstar\csname endalign*\endcsname
\renewenvironment{align*}{\linenomathNonumbers\oldalignstar}{\oldendalignstar\endlinenomath}

\usepackage{tikz-cd}




\newcounter{intro}
\newcounter{result}

\newtheorem{introthm}[intro]{Theorem}

\theoremstyle{definition}
\newtheorem{theorem}[result]{Theorem}
\newtheorem{lemma}[result]{Lemma}

\newtheorem{proposition}[result]{Proposition}
\newtheorem{corollary}[result]{Corollary}
\newtheorem{definition}[result]{Definition}

\newtheorem{remark}[result]{Remark}

\newtheorem{question}[result]{Question}

\numberwithin{equation}{subsection}

\title[A note on stability conditions on product-quotient surfaces]{A note on stability conditions on product-quotient surfaces}

\author[U.~Dutta]{Uttaran Dutta}
\address{Department of Mathematics, University of South Carolina}
\email{udutta@email.sc.edu}

\date{\today}

\definecolor{bleudefrance}{rgb}{0.19, 0.55, 0.91}

\def\P{\mathbb{P}}
\def\C{\mathbb{C}}

\def\Q{\mathbb{Q}}

\def\L{\mathbb{L}}
\def\H{\mathbb{H}}
\def\Z{\mathbb{Z}}
\def\ZZ{\mathcal{Z}}

\def\TT{\mathcal{T}}
\def\PP{\mathcal{P}}
\def\UU{\mathcal{U}}
\def\R{\mathbb{R}}
\def\OO{\mathcal{O}}
\def\FF{\mathcal{F}}
\def\AA{\mathcal{A}}

\def\OO{\mathcal{O}}

\DeclareMathOperator{\Hilb}{Hilb}

\DeclareMathOperator{\Forg}{Forg}

\DeclareMathOperator{\ch}{ch}
\DeclareMathOperator{\Alb}{Alb}
\DeclareMathOperator{\Db}{\mathrm{D}^b}
\DeclareMathOperator{\Dbg}{\mathrm{D}_G^b}
\DeclareMathOperator{\Ubg}{\UU_{\mathrm{BG}}}
\DeclareMathOperator{\GHilb}{G\text{-}\!\Hilb}

\DeclareMathOperator{\Aut}{Aut}
\DeclareMathOperator{\rk}{rk}
\DeclareMathOperator{\Stabgeo}{\Stab^{\mathrm{geo}}}

\DeclareMathOperator{\pr}{pr}
\DeclareMathOperator{\Hom}{Hom}

\DeclareMathOperator{\Pic}{Pic}

\DeclareMathOperator{\NS}{NS}

\DeclareMathOperator{\sgn}{sgn}

\DeclareMathOperator{\Coh}{Coh}
\DeclareMathOperator{\sHom}{\mathcal{H}\kern -.5pt\mathit{om}}
\DeclareMathOperator{\sTor}{\mathcal{T}\kern -1.5pt\mathit{or}}
\DeclareMathOperator{\GL}{GL}

\DeclareMathOperator{\SL}{SL}

\DeclareMathOperator{\Stab}{Stab}

\DeclareMathOperator{\td}{td}

\DeclareMathOperator{\Knum}{K_{\mathrm{num}}}

\begin{document}

\maketitle
\begin{abstract}
Let $X$ be a product-quotient surface, which is by definition a minimal resolution of $(C_1\times C_2)/G$, where $C_i$ are smooth curves of genus $\geq 2$ and $G$ is a finite group. When $G$ acts freely, we show that every product Bridgeland stability condition on $C_1\times C_2$ induces a tilt stability condition (\S \ref{subsection:tilting}) on $X$, and that, if $p_g(X)=q(X)=0$, every tilt stability condition on $X$ arises this way. When $(C_1\times C_2)/G$ has only $A_1$-singularities, we use induction \cite{MMS} and derived McKay correspondence \cite{BKR} to produce non-geometric stability conditions $\sigma_X$ on $X$, which satisfy the full support property and lie in the closure of the geometric chamber. Whether $\sigma_X$ also lies in the closure of the $\widetilde{\mathrm{GL}}^+(2,\mathbb{R})$-orbit of the tilt stability conditions depends on a single inequality, $\omega_X^2 \ge n/8$, where $n$ is the number of $A_1$-singularities. When it fails, the geometric chamber is strictly larger than this orbit.
\end{abstract}
\setcounter{tocdepth}{1}
\numberwithin{result}{section}
\section{Introduction}
Let $X$ be a smooth projective surface. The space $\Stab(X)$ of Bridgeland stability conditions is fully understood in very few cases. The work \cite{Dell} describes the set of geometric stability conditions $\Stabgeo(X).$ Thus the situation $\Stab(X)=\Stabgeo(X)$ is very desirable. Fu--Li--Zhao \cite{Fu-Li-Zhao} proved that this happens whenever $X$ has finite Albanese morphism. Following \cite{Dell}, we further explore the question: if $X$ has non-finite Albanese morphism, does there always exist a non-geometric stability condition on $X$?

Suppose a surface has finite Albanese morphism $\pi:X\to\Alb(X).$ The image has two linearly independent local $1$-forms, which can be extended to global $1$-forms $w_1,w_2$, since $\Alb(X)$ is an abelian variety. Then $\pi^*w_1\wedge\pi^*w_2$ is a global $2$-form on $X,$ which implies the canonical sheaf of $X$ has a section, i.e. $p_g\neq 0.$ Therefore, surfaces with $p_g=0$ have non-finite Albanese morphism. This brings our attention to a big class of general type surfaces, called \emph{product-quotient} surfaces. These are minimal resolutions of $(C_1\times C_2)/G$, where $C_i$ are higher genus curves and $G$ is a finite group acting faithfully on them. In this paper, we follow the classification of product-quotient surfaces with $p_g=q=0$ \cite{BP12}. 

When $G$ acts freely on $C_1\times C_2$, the equivariant derived category $\Dbg(C_1\times C_2)$ is equivalent to $\Db(X).$ The product $C_1\times C_2$ carries Liu's product stability conditions \cite{Liu-product}, which induces a stability condition on $\Db(X)$ \cite{MMS}. It is then natural to ask where the induced stability condition lands. The answer is that they are not new points in $\Stab(X).$ For $\omega,B\in\NS(X)_\R$, with $\omega$ ample, denote by $\sigma_{\omega,B}$ the tilt stability condition \cite{Bri-K3,AB13}.

\begin{introthm} (Theorem \ref{Thm:free case})
    Let $G$ act freely on $C_1\times C_2$. Every product stability condition on $C_1\times C_2$ is $G$-invariant and the induced stability condition on $\Db(X)$ is $\sigma_{\omega,B}$ for some $\omega,B\in\NS(X)_\R$, with $\omega$ ample. If $p_g=q=0,$ every tilt stability condition on $X$ arises in this way up to the $\C$-action.
\end{introthm}

When $G$ acts on $C_1\times C_2$ with fixed points, the quotient only has canonical singularities, and \cite{BKR} gives the equivalence $\Phi:\Db(X)\xrightarrow{\sim}\Dbg(C_1\times C_2)$ (Lemma \ref{lem:BKR equivalence PQ}). An induced stability condition is not geometric, but is a limit of geometric stability conditions, and there is a single numerical invariant that decides whether it is also the limit of tilt stability conditions. We prove the following general statement. 

\begin{introthm}\label{thm:B} (Theorem \ref{maintheorem}, Theorem \ref{thm:maintheorem2})
    Let $Y$ be a smooth projective surface, $G$ a finite group acting on $Y$, such that the canonical bundle of $Y$ is locally trivial as a $G$-sheaf and $Y/G$ has only $A_1$-singularities. Set $\pi_X:X\to Y/G$ to be the minimal resolution. For $\overline{\omega},\overline{B}\in\NS(Y)_\R$, $\overline{\omega}$ ample, the tilt stability condition $\sigma_{\omega,B}$ is $G$-invariant, where $\omega,B$ are pullbacks of $\overline{\omega},\overline{B}$ via the quotient $Y\to Y/G$. Then induction \cite{MMS}, followed by \cite{BKR} produces a stability condition $\sigma_X\in\Stab(X),$ such that,
    \begin{enumerate}
        \item $\sigma_X$ is non-geometric, but $\sigma_X\in \overline{\Stabgeo(X)}.$
        \item $\sigma_X$ lies in the closure of $\Ubg=\{\sigma_{\omega,B}\cdot g;\ \omega,B\in\NS(X)_\R,\text{ $\omega$ ample, }g\in\widetilde{\GL}_2^+(\R)\}$ if and only if $\omega_X^2\geq n/8,$ where $\omega_X=\pi_X^*\overline{\omega}$ and $n=\#$ of $A_1$-singularities of $Y/G.$ 
    \end{enumerate}
\end{introthm}
Tramel-Xia \cite{Tramel-Xia} constructed non-geometric stability conditions on surfaces with negative self-intersection curves. Later, \cite{Chou}, \cite{Vilches} defined stability conditions on surfaces that are minimal resolutions of ADE-singularities. The stability conditions of Theorem \ref{thm:B} are of a different nature; they are induced from the cover rather than by tilting on $X.$ They lie in the closure of $\Stabgeo(X)$, but when $\omega_X^2<n/8,$ they lie outside of $\overline{\Ubg},$ and thus $\Stabgeo(X)$ is strictly larger than the $\widetilde{\GL}_2^+(\R)$-orbit of the tilt stability conditions.

\subsection*{Acknowledgements}
During the HIMR Focused Research Grant collaborative visit \emph{Logarithmic Geometry and Smoothing Bridgeland Stability Conditions} in Cambridge, Hannah Dell raised the question of non-finite Albanese morphisms, which led the author to look for further examples. The author thanks the Heilbronn Institute for Mathematical Research (HIMR) for supporting this visit, Nicol\'as Vilches and Xiaolei Zhao for pointing out the reference \cite{Chou}, and Hannah Dell, Howard Nuer, Alekos Robotis, and Nicol\'as Vilches for their comments on an earlier draft. The author was partially supported by a Simons Dissertation Fellowship in Mathematics.

\subsection*{Conventions}
We work over $\C$. $\NS(X)$ is the N\'eron--Severi group, and $\NS(X)_\R=\NS(X)\otimes\R.$ $\Knum(X)$ is the numerical Grothendieck group; which is defined as $K_0(X)/T$, where $T$ is the set of $v\in K_0(X)$ such that $\chi(v,w)=0$ for all $w\in K_0(X).$ This is a finitely generated $\Z$-lattice. $\Db(X)$ is the bounded derived category of coherent sheaves.

\section{Preliminaries}
\subsection*{Bridgeland Stability Conditions}
In this section, we recall some basic definitions and properties of Bridgeland stability conditions on $\Db(X)$, the bounded derived category of coherent sheaves on a smooth projective variety $X$ over $\mathbb{C}$. For more details, we refer to \cite{Bri-original}, \cite{Macri-Schmidt}.

\begin{definition}\label{Def:pre-stability}
    A \emph{slicing} $\PP$ on $\Db(X)$ is a collection of full additive subcategories $\PP(\phi)\subset\Db(X)$ for $\phi\in\R$, where
    \begin{itemize}
        \item $\PP(\phi+1)=\PP(\phi)[1].$
        \item If $\phi_1>\phi_2,$ $E_i\in\PP(\phi_i),$ then $\Hom(E_1,E_2)=0.$
        \item For all $E\in\Db(X),$ there are real numbers $\phi_1>\dots>\phi_m,$ objects $E_i\in\Db(X),$ for $i=1,\dots,m$ and a collection of triangles
       \[\begin{tikzcd}
	{0=E_0} & {E_1} & {E_2} & \dots & {E_{m-1}} & {E_m=E} \\
	& {A_1} & {A_2} && {A_{m-1}} & {A_m}
	\arrow[from=1-1, to=1-2]
	\arrow[from=1-2, to=1-3]
	\arrow[from=1-2, to=2-2]
	\arrow[from=1-3, to=1-4]
	\arrow[from=1-3, to=2-3]
	\arrow[from=1-4, to=1-5]
	\arrow[from=1-5, to=1-6]
	\arrow[from=1-5, to=2-5]
	\arrow[from=1-6, to=2-6]
	\arrow[dashed, from=2-2, to=1-1]
	\arrow[from=2-3, to=1-2]
	\arrow[dashed, from=2-5, to=1-4]
	\arrow[dashed, from=2-6, to=1-5]
\end{tikzcd}\]
where $A_i\in\PP(\phi_i).$ For $E\in\PP(\phi)$, we call the $\phi(E):=\phi$ the \emph{phase} of $E.$
    \end{itemize}

A \emph{pre-stability condition} on $\Db(X)$ is a pair $(Z,\PP),$ with $Z:K_0(X)\to\C$ is an additive group homomorphism, called \emph{central charge} and $\PP$ a slicing, such that,
\[
Z(E)=m(E)\,e^{i\pi\phi}\quad\text{with } m(E)\in\R_{>0}\qquad\text{for every }0\neq E\in\PP(\phi).
\]
Here $K_0(X)$ is the Grothendieck group of $X.$
\end{definition}

\begin{definition}
    Let $\Lambda$ be a finite-rank lattice with a surjective group homomorphism $K_0(X)\xrightarrow{\lambda} \Lambda$.
\begin{enumerate}
  \item A pre-stability condition $\sigma = (\PP, Z)$ on $\Db(X)$ satisfies the \emph{support property} with respect to $(\Lambda, \lambda)$ if
  \begin{enumerate}
    \item $Z$ factors via $\Lambda$, \emph{i.e.} $Z: K_0(X) \xrightarrow{\lambda} \Lambda \to \C$, and
    \item there exists a quadratic form $Q$ on $\Lambda \otimes \R$ such that
    \begin{enumerate}
      \item $\ker Z \otimes \R$ is negative definite with respect to $Q$, and
      \item every $\sigma$-semistable object $E \in \Db(X)$ satisfies $Q(\lambda(E)) \geq 0$.
    \end{enumerate}
  \end{enumerate}
  \item A Bridgeland pre-stability condition $\sigma$ on $\Db(X)$ that satisfies the support property with respect to $(\Lambda, \lambda)$ is called a \emph{Bridgeland stability condition} (with respect to $(\Lambda, \lambda)$). If $\lambda$ also factors via $\Knum(X)$, we call $\sigma$ a \emph{numerical Bridgeland stability condition}.
\end{enumerate}
\end{definition}

Unless stated otherwise, we will assume that all Bridgeland stability conditions are numerical. The set of numerical stability conditions on $\Db(X)$ will be denoted $\Stab(X).$

\begin{theorem}\cite{Bri-original}
    The space of stability conditions $\Stab(X)$ has the natural structure of a complex manifold of dimension $\rk(\Knum(X))$. The forgetful map $\ZZ$ defines the local homeomorphism
    \begin{align*}
        \ZZ:\Stab(X)&\longrightarrow \Hom(\Knum,\C),\\
        (Z,\PP)&\longmapsto Z.
    \end{align*}
\end{theorem}

There is an equivalent characterisation of Bridgeland stability conditions, which uses the notion of a $t$-structure on a triangulated category.

\begin{definition}
    The heart of a bounded $t$-structure in $\Db(X)$ is a full additive subcategory $\AA$ such that
    \begin{itemize}
        \item $\Hom_{\Db(X)}(\AA[k_1],\AA[k_2])=0$ for $k_1>k_2.$
        \item For any object $E\in\Db(X)$ there are integers $k_1>\dots>k_m$ and a collection of exact triangles
        \[\begin{tikzcd}
	{0=E_0} & {E_1} & {E_2} & \dots & {E_{m-1}} & {E_m=E} \\
	& {A_1} & {A_2} && {A_{m-1}} & {A_m}
	\arrow[from=1-1, to=1-2]
	\arrow[from=1-2, to=1-3]
	\arrow[from=1-2, to=2-2]
	\arrow[from=1-3, to=1-4]
	\arrow[from=1-3, to=2-3]
	\arrow[from=1-4, to=1-5]
	\arrow[from=1-5, to=1-6]
	\arrow[from=1-5, to=2-5]
	\arrow[from=1-6, to=2-6]
	\arrow[dashed, from=2-2, to=1-1]
	\arrow[from=2-3, to=1-2]
	\arrow[dashed, from=2-5, to=1-4]
	\arrow[dashed, from=2-6, to=1-5]
\end{tikzcd}\]
where $A_i\in\AA[k_i]$ for $i=1,\dots,m.$
    \end{itemize}
\end{definition}

\begin{definition}\cite{Bri-original}
Let $\AA$ be an abelian category. A \emph{stability function} for $\AA$ is a group homomorphism $Z\colon K(\AA) \to \C$ such that for every non-zero object $E$ of $\AA$,
\[
Z([E]) \in \H := \{ m \cdot e^{i\pi\phi} \mid m \in \R_{>0},\ \phi \in (0,1] \} \subset \C.
\]
For every non-zero object $E$, we define the \emph{phase} by $\phi(E) = \frac{1}{\pi}\arg(Z([E])) \in (0,1]$. We say an object $E$ is \emph{$Z$-stable} (resp.\ \emph{$Z$-semistable}) if $E \neq 0$ and for every proper non-zero subobject $A$, we have $\phi(A) < \phi(E)$ (resp.\ $\phi(A) \leq \phi(E)$).
\end{definition}

\begin{definition}\cite{Bri-original}
Let $\AA$ be an abelian category, and let $Z\colon K(\AA) \to \C$ be a stability function on $\AA$. A \emph{Harder--Narasimhan (HN) filtration} of a non-zero object $E$ of $\AA$ is a finite chain of subobjects in $\AA$,
\[
0 = E_0 \subset E_1 \subset \cdots \subset E_{n-1} \subset E_n = E,
\]
such that each factor $F_i = E_i/E_{i-1}$ (called a \emph{Harder--Narasimhan factor}) is a $Z$-semistable object of $A$ and $\phi(F_1) > \phi(F_2) > \cdots > \phi(F_n)$. Moreover, we say that $Z$ has the \emph{Harder--Narasimhan property} if every non-zero object of $\AA$ has a Harder--Narasimhan filtration.
\end{definition}

\begin{proposition}\cite{Bri-original}
To give a Bridgeland pre-stability condition $(\PP, Z)$ on $\Db(X)$ is equivalent to giving a pair $(Z_{\AA}, \AA)$, where $\AA$ is the heart of a bounded $t$-structure on $\Db(X)$ and $Z_{\AA}$ is a stability function for $\AA$ which has the Harder--Narasimhan property.

Moreover, $(\PP, Z)$ is a numerical Bridgeland stability condition if and only if $Z_{\AA}$ factors via $\Knum(X)$ and satisfies the support property for $Z_{\AA}$-semistable objects.
\end{proposition}

\subsubsection*{Group actions}
The autoequivalence group of $\Db(X)$, $\Aut(\Db(X))$ acts on the left by $\Phi\cdot(Z,\PP)=(Z\circ\Phi_*^{-1},\,\Phi(\PP))$, where $\Phi(\PP)(\phi)=\Phi(\PP(\phi))$. The universal cover $\widetilde{\GL}^+(2,\R)$ of $\GL^+(2,\R)$ acts on the right (see \cite{Macri-Schmidt} for details). These two actions commute. As a subgroup $\C\subset\widetilde{\GL}^+(2,\R)$ acts on $\Stab(X)$: $\lambda\in\C$ sends $(Z,\PP)$ to $(e^{-i\pi\lambda}Z,\ \PP')$ with $\PP'(\phi)=\PP(\phi+\Re\lambda)$. 

\begin{definition}
    A stability condition $\sigma\in\Stab(X)$ is \emph{geometric} if every skyscraper sheaf $\OO_x$, $x\in X$, is $\sigma$-stable and all $\OO_x$ have the same phase. We write $\Stabgeo(X)\subset\Stab(X)$ for the set of geometric stability conditions.
\end{definition}
$\Stabgeo(X)$ is stable under the action of $\widetilde{\GL}^+(2,\R)$, and under the subgroup of $\Aut(\Db(X))$ generated by automorphisms of $X$ and twists by line bundles. Using the $\C$-action one can always normalize a geometric stability condition so that every $\OO_x$ has phase $1$ and $Z(\OO_x)=-1$.

\subsubsection{Geometric stability conditions on surfaces}\label{subsection:tilting}
Let $X$ be a smooth projective surface, and $\omega,B\in\NS(X)_\R$ with $\omega$ ample. For any coherent sheaf $E\in\Coh(X)$ the twisted Mumford slope function is defined as follows
\[
\mu_{B,\omega}(E)=
\begin{cases}
\dfrac{\omega\cdot\ch_1^B(E)}{\omega^2\,\ch_0(E)}, & \text{if } \operatorname{ch}_0(E)\neq 0,\\[2ex]
    +\infty, & \text{if } \ch_0(E)= 0.
\end{cases}
\]
Here $\ch^B(E)=e^{-B}\ch(E)$ is the $B$-twisted Chern character of $E.$ We define a pair of subcategories of $\Coh(X),$
\begin{align*}
    \FF_{\omega,B}&=\{E\in\Coh(X)\,;\,\text{any semistable factor $F$ of $E$ satisfies }\mu_{B,\omega}(F)\leq 0\},\\
    \TT_{\omega,B}&=\{E\in\Coh(X)\,;\,\text{any semistable factor $F$ of $E$ satisfies }\mu_{B,\omega}(F)> 0\}.
\end{align*}
Then the smallest extension closed subcategory $\AA_{\omega,B}=\langle\FF_{\omega,B}[1],\TT_{\omega,B} \rangle_{\mathrm{ext}}$ forms the heart of a $t$-structure (see \cite[\S 6]{Macri-Schmidt}). Define the central charge
\[
Z_{\omega,B}(E)=-\int_X e^{-(B+i\omega)}\ch(E).
\]
\begin{theorem}\cite{Bri-K3}, \cite{AB13}
    On a smooth projective surface $X,$ the pair $(Z_{\omega,B},\AA_{\omega,B})$ forms a geometric stability condition for any choice of $\omega,B\in\NS(X)_\R$, where $\omega$ is ample. 
\end{theorem}
The heart $\AA_{\omega,B}$ is formed via a \emph{tilting construction}. Hence we call these \emph{tilt-stability conditions} on $X.$

\subsection*{Product-Quotient Surfaces}
Let $G$ be a finite group acting faithfully on two smooth projective curves $C_1$ and $C_2$ of genus at least $2$ and diagonally on $C_1\times C_2,$ i.e. $g\cdot (x,y)=(g\cdot x,g\cdot y).$
\begin{definition}
    The minimal resolution $X$ of the singularities of $(C_1\times C_2)/G$ is called a \emph{product-quotient} surface.
\end{definition}

These surfaces form an interesting class of general type surfaces. We refer to \cite{BP12} for the classification of minimal product-quotient surfaces with $p_g=q=0$. There are $72$ families of product-quotient surfaces, in which $13$ families of surfaces have the group $G$ acting freely on $C_1\times C_2.$ In these cases, we have $X=(C_1\times C_2)/G.$ Whenever $G$ acts on $C_1\times C_2$ with fixed points, the quotient $(C_1\times C_2)/G$ has canonical singularities. In the second part of the paper we explore stability conditions on the surfaces where $(C_1\times C_2)/G$ has only $A_1$-singularities.

If $G$ acts freely, we have an equivalence of derived categories,
\[
\Db(X)\cong\Dbg(C_1\times C_2),
\]
where $\Dbg(C_1\times C_2)$ is the derived category of $G$-equivariant sheaves on $C_1\times C_2.$ There are two seemingly different constructions of geometric stability conditions on $X$. One is the standard tilt stability construction on a smooth projective surface described above. The other is a \emph{product} stability condition on $C_1\times C_2$ defined by \cite{Liu-product} that induces a geometric stability condition on $\Dbg(C_1\times C_2),$ and hence on $\Db(X)$. In the next section, we prove that these induced stability conditions do not leave the tilt stability chamber.

\section{Free action}
Throughout this section $C_1,C_2$ are smooth projective curves of genus at least $1$, $S=C_1\times C_2$, and $G$ is a finite group that acts faithfully on $C_1$ and on $C_2$, diagonally on $S$. Let $\operatorname{pr}_i\colon S\to C_i$ be the projections and
\[
f_1=\operatorname{pr}_1^*[\mathrm{pt}],\qquad f_2=\operatorname{pr}_2^*[\mathrm{pt}]
\]
be the fibre classes. Both classes are $G$-invariant and we have $f_1^2=f_2^2=0$ and $f_1\cdot f_2=1$.

\subsection{Central charges on curves and on the product}
Following \cite{Curves-Macri}, we know that the stability manifold $\Stab(C)$ of a smooth projective curve $C$ of genus at least $1$ is a single orbit $[(Z_0,\Coh(C))]$ under the action of $\widetilde{\operatorname{GL}}_2^+(\R)$, where $Z_0(E)=-\deg E+i\operatorname{rk}E$ for any object $E\in \Db(C).$

\begin{lemma}\label{stability-on-curves}
    Every central charge on $C$ has the form
\[
Z^\gamma(E)=-\deg E+\gamma\operatorname{rk}E,
\]
up to the $\C$-action, where $\Im\gamma>0.$
\end{lemma}
\begin{proof}
    An element $A=\begin{pmatrix}
        p & q\\
        r & s
    \end{pmatrix} \in\GL_2^+(\R)$ acts on $\C\cong\R^2$ via,
    \[
    f_A(x+iy)=(px+qy)+i(rx+sy).
    \]
    On central charges, the action of $A$ is $Z\mapsto f_A^{-1}\circ Z$. As $A\mapsto A^{-1}$ is a bijection of $\GL_2^+(\R)$, the orbit of $Z_0$ is $\{f_A\circ Z_0: A\in\GL_2^+(\R)\}$, and
\[
(f_A\circ Z_0)(E)=(-p-ir)\deg E+(q+is)\rk E.
\]

    Thus, every central charge on $C$ has the form,
    \[
    Z(E)=\alpha\deg E+\beta\rk E,
    \]
    where $\alpha=-p-ir$ and $\beta=q+is$. We have
    \[
    \det\begin{pmatrix}
        p & q\\
        r & s
    \end{pmatrix}>0,
    \]
    in other words, $ps-qr>0.$ This implies,
    \[
    \Im(\gamma)=\Im\left(\frac{-\beta}{\alpha}\right)=\frac{ps-qr}{p^2+r^2}>0.
    \]
    Upto the $\C$-action, we can say that $Z(E)$ takes the form,
\[
Z^\gamma(E)=-\deg E+\gamma\rk E,
\]
where $\Im\gamma>0.$
\end{proof}

In \cite{Liu-product}, Y. Liu defined the existence of Bridgeland stability conditions on product of a variety with a curve, given a stability condition on the
variety itself. Below we follow the description given in \cite{LMPSZ}.

\begin{lemma}
    Fix $(B,\omega)\in \Q\times\Q_{>0}$. Then,
    the product central charge $Z_S^{B,\omega}$ takes the following form 
    \[
    Z_S^{B,\omega}(E_1\boxtimes E_2)=-Z_{C_1}^{B+i\omega}(E_1)Z_{C_2}^{B+i\omega}(E_2).
    \]
    \end{lemma}
\begin{proof}
Following \cite[Eq~(4.10) and Lemma~4.4]{LMPSZ}, the product stability condition on
$S=C_1\times C_2$ has central charge
\[
Z_S^{B,\omega}(E)
=
-\int_S e^{-(B+i\omega)(f_1+f_2)}\operatorname{ch}(E),
\]
$E_1\boxtimes E_2=\pr_1^*(E_1)\otimes\pr_2^*(E_2)$;  write
\[
r_i=\operatorname{rk}(E_i),
\qquad
d_i=\deg(E_i).
\]
Since
\[
e^{-(B+i\omega) f_i}=1-(B+i\omega)f_i,
\]
we get
\[
e^{-(B+i\omega)(f_1+f_2)}\operatorname{ch}(E_1\boxtimes E_2)
=
\bigl(r_1+(d_1-(B+i\omega)r_1)f_1\bigr)
\bigl(r_2+(d_2-(B+i\omega)r_2)f_2\bigr).
\]
Hence the degree-two component is
\[
(d_1-(B+i\omega)r_1)(d_2-(B+i\omega)r_2)f_1f_2=(d_1-(B+i\omega)r_1)(d_2-(B+i\omega)r_2).
\]
It follows that
\[
Z_S^{B,\omega}(E_1\boxtimes E_2)
=
-(d_1-(B+i\omega)r_1)(d_2-(B+i\omega)r_2).
\]
Finally, since
\[
Z_{C_i}^{B+i\omega}(E_i)=-d_i+(B+i\omega)r_i,
\]
we conclude that
\[
Z_S^{B,\omega}(E_1\boxtimes E_2)
=
-Z_{C_1}^{B+i\omega}(E_1)\,
 Z_{C_2}^{B+i\omega}(E_2).
\]
\end{proof}

Now for $\omega,B\in\operatorname{NS}(S)_{\R}$ with $\omega$ ample define the central charge as in \S \ref{subsection:tilting},

\[
Z_{\omega,B}(E)=-\int_S e^{-(B+i\omega)}\operatorname{ch}(E)
\]

Start with the central charges $Z_{C_1}^{\gamma_1}=-\deg+\gamma_1\rk$ and $Z_{C_2}^{\gamma_2}=-\deg+\gamma_2\rk$ on the curves $C_1$ and $C_2$ respectively, with $\Im\gamma_i>0$. Motivated from the previous lemma, we denote the product stability constructed from $Z_1$ and $Z_2$ as $Z_1\otimes Z_2.$

\begin{lemma}\label{central-charge}
Let $\omega,B$ be the classes with
\[
B+i\omega=\gamma_1f_1+\gamma_2f_2.
\]
Then $\omega$ is ample and
\[
Z_{C_1}^{\gamma_1}\otimes Z_{C_2}^{\gamma_2}=-\alpha_1\alpha_2\,Z_{\omega,B},
\]
where $\gamma_i=-\frac{\beta_i}{\alpha_i}$ as in the proof of Lemma \ref{stability-on-curves}.
\end{lemma}

\begin{proof}
A divisor $af_1+bf_2$ on $S$ is ample if and only if $a,b>0$, and $\operatorname{Im}\gamma_i>0$ implies $\omega$ is ample. Rest of the proof goes similarly as the previous lemma.
\end{proof}

\begin{proposition}\label{geometric-stab}
\cite[Prop. 10.3]{Bri-K3} Let $X$ be a smooth projective surface and $\sigma=(Z,\mathcal A)$ a stability condition on $D^b(X)$ such that $Z=Z_{\omega,B}$ for some ample $\omega$ and some $B\in\operatorname{NS}(X)_{\mathbb R}$, and such that every skyscraper sheaf is $\sigma$-stable of phase $1$. Then $\mathcal A=\operatorname{Coh}^{\omega,B}(X)$, and $\sigma=\sigma_{\omega,B}$.
\end{proposition}

\begin{corollary}\label{product=tilt}
    There is $\lambda\in\C$ with $e^{-i\pi\lambda}=-\alpha_1\alpha_2$ such that
    $\sigma^{\gamma_1}_{C_1}\otimes\sigma^{\gamma_2}_{C_2}=\lambda\cdot\sigma_{\omega,B}.$
\end{corollary}
\begin{proof}
    By Lemma~\ref{central-charge}, $\sigma^{\gamma_1}_{C_1}\otimes\sigma^{\gamma_2}_{C_2}$ has central charge $-\alpha_1\alpha_2 Z_{\omega,B}$, and skyscraper sheaves are stable, all of one phase $\phi$. Take $\lambda$ with $e^{-i\pi\lambda}=-\alpha_1\alpha_2$ and $\Re\lambda=1-\phi$. Then $(-\lambda)\cdot(\sigma^{\gamma_1}_{C_1}\otimes\sigma^{\gamma_2}_{C_2})$ has central charge $Z_{\omega,B}$ and slicing $\PP''(1)=\PP(1-\Re\lambda)=\PP(\phi)$. Thus skyscraper sheaves stable of phase $1$, so it equals $\sigma_{\omega,B}$ by Prop.~\ref{geometric-stab}.
\end{proof}

In particular, the product stability conditions lie in the $\C$-orbits of tilt stability conditions.

\subsection{Inducing stability conditions along a free quotient}
Now let $G$ act freely on $S=C_1\times C_2$ and let $X=S/G$ with $\pi:S\to X$ the quotient map. Since the action is free $\pi$ is \'etale of degree $|G|$. In this case, we get the equivalence,
\[
\Db(X)\cong \Dbg(S),
\]
and $\Forg:\Dbg(S)\to\Db(S)$ is just $\pi^*.$ In \cite{MMS}, the authors proved that 
\[
\sigma_G=(Z\circ\operatorname{Forg},\ \mathcal P_G),\qquad \mathcal P_G(\phi)=\{F\in \Dbg(S):\operatorname{Forg}(F)\in\mathcal P(\phi)\},
\]
is a stability condition on $\Dbg(S)$.

\begin{proposition}\label{prop:free}
Let $\operatorname{NS}(S)^G_{\R}=(\operatorname{NS}(S)_{\R})^G$ be the $G$-invariant part of the N\'eron-Severi space of $S$ and let $\omega,B\in\operatorname{NS}(S)_{\mathbb R}^G$ where $\omega$ is ample. Take $\omega_X,B_X\in\operatorname{NS}(X)_{\mathbb R}$ the induced classes, i.e. $\pi^*\omega_X=\omega$, $\pi^*B_X=B$. Then $\sigma_{\omega,B}$ is $G$-invariant, and under the equivalence $\Db(X)\cong\Dbg(S),$
\[
(\sigma_{\omega,B})_G=|G|\cdot\sigma_{\omega_X,B_X}.
\]
\end{proposition}

\begin{proof}
Since $\pi$ is finite, $\pi^*\omega_X=\omega$ is ample, and in fact $\pi^*:\operatorname{NS}(X)_{\mathbb R}\to\operatorname{NS}(S)_{\mathbb R}^G$ is an isomorphism for a finite free quotient. Thus, for $g\in G$, $g^*$ preserves $\omega,B$, and thus $\sigma_{\omega,B}$ is $G$-invariant. 

Since $\pi$ is \'etale, hence flat, $\pi^*$ is exact and $\L\pi^*E=\pi^*E$ for any $E\in\Db(X)$. The projection formula gives,
\begin{align*}
    Z_{\omega,B}(\pi^*E)&=-\int_S e^{-B-i\omega}\ch(\pi^*E)\\
    &=-\int_S e^{-B-i\omega}\pi^*\ch(E)\\
    &=-\int_S \pi^*\left(e^{-B_X-i\omega_X}\ch(E))\right)\\
    &=-|G|\int_S e^{-B_X-i\omega_X}\ch(E)) = |G|Z_{\omega_X,B_X}(E).
\end{align*}
For $x\in X,$ $\Forg(\OO_x)=\pi^*(\OO_x)=\oplus_{p\in\pi^{-1}(x)}\OO_p$. $\OO_p$ is $\sigma_{\omega,B}$-stable of phase $1,$ so $\pi^*\OO_x\in\PP(1)$ and by definition $\OO_x\in\PP_G(1).$ In other words $\OO_x$ is $(\sigma_{\omega,B})_G$-semistable of phase $1,$ and a destabilizing subobject must share the same phase. Suppose $E\subset\OO_x$ is a proper non-zero subobject of $\OO_x$ in $\PP_G(1).$ Then $\Forg(E)\subseteq \oplus_{p\in\pi^{-1}(x)} \OO_p$ in $\PP(1).$ Each $\OO_p$ is stable in $\PP(1),$ thus simple. This implies a subobject $\Forg(E)$ is of the form $\oplus_{p\in T}\OO_p$, where $T\subseteq\pi^{-1}(x).$ But notice that the underlying support $T$ must be $G$-invariant. The set $\pi^{-1}(x)$ is one free $G$-orbit, so $G$ acts transitively on it. This implies $T=\emptyset$ or $T=\pi^{-1}(x).$ But $E$ is non-empty and a proper subset, so both the possibilities are ruled out. So $\OO_x$ is $\sigma_G$-stable of phase one $\forall x\in X.$

Thus the stability condition $|G|^{-1}\cdot(\sigma_{\omega,B})_G$ on $\Db(X)$ has the central charge $Z_{\omega_X,B_X}$ and all skyscraper sheaves are stable of phase $1,$ so we conclude it is equal to $\sigma_{\omega_X,B_X}$ by proposition \ref{geometric-stab}.
\end{proof}

\begin{theorem}\label{Thm:free case}
    Let $G$ act freely on $C_1\times C_2$ with $X=(C_1\times C_2)/G$. Every product stability condition on $C_1\times C_2$ is $G$-invariant, and the induced stability condition on $\Db(X)$ is $\sigma_{\omega_X,B_X}$ up to the action of $\C$ for some $\omega_X,B_X\in \NS(X)_\R$, with $\omega_X$ ample. 
    
    Moreover, if $p_g(X)=q(X)=0$, then every tilt stability condition $\sigma_{\omega,B}$ on $X$ arises in this way.
\end{theorem}
\begin{proof}
    Let $\sigma_i\in\Stab(C_i)$ and $\sigma_1\otimes\sigma_2\in\Stab(C_1\times C_2)$ be the product. Corollary \ref{product=tilt} implies $\sigma_1\otimes\sigma_2 =\lambda\cdot\sigma_{\omega,B}$ with $\omega,B\in \R f_1\oplus\R f_2$. Since these classes are $G$-invariant, so is $\sigma_1\otimes\sigma_2$. By Proposition \ref{prop:free}, the induced stability condition is $\lambda\, |G|\cdot\sigma_{\omega_X,B_X},$ hence the result follows.
    
    Since $p_g(X)=q(X)=0$, we have $\chi(\OO_X)=1$, so Noether's formula gives $e(X)=12\chi(\OO_X)-K_X^2=4$ (following \cite{BP12}, $K_X^2=8$). As $b_1(X)=2q(X)=0$ and $b_3(X)=b_1(X)$ by Poincaré duality, $b_2(X)=e(X)-2=2$. Now, $p_g(X)=q(X)=0$; thus the exponential sequence gives $\Pic(X)\cong H^2(X,\Z)$; in particular $\rho(X)=b_2(X)=2$. Let $F_i\in\NS(X)$ be the classes of fibers of the fibrations $X\to C_i/G,$ for $i=1,2.$ On $S,$ $\pi^*F_i$ are positive multiples of $f_i.$ This gives $F_1\cdot F_2>0$ and $F_i^2=0.$ Hence, $F_1, F_2$ are linearly independent in $\NS(X)_{\R}.$ This implies $\NS(X)_{\R} =\R F_1\oplus\R F_2.$ By Lemma \ref{central-charge}, we can choose $\gamma_i\in\mathbb{H}$ and produce any pair $(\omega_X,B_X)\in\R F_1\oplus\R F_2,$ and hence, every $\sigma_{\omega_X,B_X}.$ 
\end{proof}

\begin{remark}
    Following \cite{Dell-Heng-Licata}, we know that $\Stabgeo(X)$ forms a connected component of $\Stab(X)$ in the free-action cases. Thus, the existence of a non-geometric stability condition would imply $\Stab(X)$ is not connected. In all of the thirteen cases, there exist max-length exceptional collections of line bundles on $\Db(X)$ \cite{Galkin-Shinder,Lee1,Lee2,Lee3,Lee4}. But the collections are not full, and the orthogonal subcategory is a \emph{quasi-phantom}. It is known that there does not exist any stability condition on a quasi-phantom. This is why we cannot glue stability conditions along these exceptional collections to construct new stability conditions. In other words, if there does exist a non-geometric stability condition on $\Db(X),$ then it cannot be constructed by gluing from the exceptional collection. On the other hand, the canonical bundle is ample; thus $\Aut(\Db(X))$ is generated by the standard autoequivalences, and these preserve $\Stabgeo(X).$ This implies there's no interesting autoequivalence whose action can give rise to non-geometric stability conditions. So a natural question to ask,
    \begin{question}
        In the cases when $G$ acts freely on $C_1\times C_2$, i.e. when $X=(C_1\times C_2)/G$, is $\Stab(X)=\Stabgeo(X)$?
    \end{question}
\end{remark}

\section{When the action has fixed points}
In this section we construct non-geometric stability conditions on $X$, when the quotient $(C_1\times C_2)/G$ has only $A_1$-singularities. Following the classification \cite[Table 1,2]{BP12}, there are more than $20$ families of surfaces that fit into this. We do not rely on the product structure anymore, and prove our results in a general setup.

\subsection{Setup - the derived McKay correspondence}
Let $Y$ be a smooth quasi-projective variety with an action of a finite group $G$. Let $\omega_Y$, the canonical sheaf of $Y$ be locally trivial as a $G$-sheaf, equivalently, the stabilizer $G_p\subset\SL(T_{p}Y)$ for every $p\in Y.$ A $G$-cluster $Z\subset Y$ is a $G$-invariant $0$-dimensional subscheme, such that $\dim_\C H^0(\OO_Z)=|G|$ and as a $G$-representation $H^0(\OO_Z)\cong \C[G],$ the regular representation. The $G$-Hilbert scheme $G\text{-}\!\Hilb(Y)$ parametrizes $G$-clusters on $Y.$ The universal family $\ZZ\subset G\text{-}\!\Hilb(Y)\times Y$ is a closed subscheme, with $\ZZ_x\subset Y$ the $G$-cluster represented by $x.$

In \cite{BKR}, under a dimension condition, the authors prove that $\GHilb(Y)\to Y/G$ is a crepant resolution and the Fourier-Mukai equivalence of derived categories,
\[
\Db(\GHilb(Y))\simeq  \Dbg(Y).
\]
When $\dim Y\leq 3,$ the dimension condition automatically follows. Recall that a pseudo-reflection of a complex vector space $V$ is a non-trivial linear auotomorphism $s\in \GL(V)$ whose fixed subspace is a hyperplane. If $\dim V=2$, then $s$  takes the form 
\[
s=\begin{pmatrix}
1&0\\
0&\lambda
\end{pmatrix}
\]
in some suitable basis, where $\lambda\neq 1.$ Hence $\det(s)=\lambda\neq 1.$

From now on, let $Y$ be a smooth projective surface and $G$ a finite group acting faithfully on $Y$ such that the canonical sheaf $\omega_Y$ is locally trivial as a $G$-sheaf and $Y/G$ has only $A_1$ singularities. 

\begin{lemma}
Every nontrivial stabilizer $G_p$ is $\mathbb Z/2$
acting on $T_pY$ by $-\mathrm{id}$, and the set $P$ of such points is finite.
\end{lemma}
\begin{proof}
    From the locally trivial hypothesis, we have $G_p\subset\SL(T_pY)$ \cite[\S 1]{BKR}. If $G_p$ contained some pseudo-reflection $g$ of $T_pY,$ that would mean $\det g\neq 1,$ contradicting $G_p\subset\SL(T_pY)$. Hence, $G_p$ contains no pseudo-reflections.

    We have the local quotient description $(Y/G,[p])\cong (T_pY/G_p,0).$ The quotient is smooth if and only if $G_p$ is generated by pseudo-reflections (\cite{Chevalley1955, ShephardTodd1954}). Because $G_p$ has no pseudo-reflections,  $T_pY/G_p$ is singular, hence by assumption an $A_1$ singularity. By the classification of du Val singularities, a finite subgroup of $\SL_2(\C)$ with quotient of type $A_1$ is cyclic of order $2$, i.e. $G_p\cong \Z/2$ and its nontrivial element acts by $(u,v)\mapsto (-u,-v)$, in other words, by $-\mathrm{id}_{T_pY}$.
\end{proof}

Let $P\subset Y$ be the set of points with non-trivial stabilizer and $P=P_1\sqcup\dots \sqcup P_n$ be the decomposition into $G$-orbits, where $|P_j|=|G|/2.$ The quotient $\pi_Y:Y\to \overline{Y}=Y/G$ has $n$ $A_1$ singularities. Write $\pi_X:X\to \overline{Y}$ for the minimal resolution and $E_j$ for the $(-2)$-curves over the images of $P_j.$ 

Take $q:\widetilde{Y}=\operatorname{Bl}_P(Y)\to Y$ to be the blow-up and write $e_p$ for the exceptional curve over $p\in P.$ Since $P$ is $G$-invariant, the $G$-action lifts to the blow-up $q:\widetilde{Y}=\operatorname{Bl}_P(Y)\to Y$; write $e_p$ for the exceptional curve over $p\in P$ and $p: \widetilde{Y}\to \widetilde{Y}/G$ for the quotient map. 

\begin{lemma}
    $\widetilde Y/G$ is smooth and the induced morphism $\widetilde Y/G\to\overline Y$ is the minimal resolution; hence $\widetilde Y/G\cong X$. The images of the exceptional curves $e_p$, for $p\in P_j$, coincide with the exceptional $(-2)$-curves $E_j$ over the corresponding $A_1$-singularities.
\end{lemma}
\begin{proof}
    The morphism $q:\widetilde{Y}\to Y$ induces a proper birational morphism $\pi_X:\widetilde{Y}/G\to \overline{Y}.$ Away from $P$ the the map $q$ is an isomorphism and $G$ acts freely. Thus $\pi_X$ is an isomorphism over the smooth locus of $\overline{Y}.$ By the previous lemma, we know that $G_p\cong\Z/2 = \{1,\iota\}$ for any $p\in P$ and $\iota\cdot(x,y)=(-x,-y)$ for some chosen coordinates near $p.$ After quotenting, we have an $A_1$ singularity, locally $\C^2/\{\pm 1\}.$ 
    
    On the chart $y=xt,$ the coordinates are $(x,t)$ and the lifted involution $(x,t)\mapsto (-x,t).$ Thus after quotienting this chart gets the $\Z/2$-invariant coordinates $(u,t)=(x^2,t)$ and is smooth. Similarly, one can show that the other chart $x=sy$ also gives another smooth affine chart. The quotient of the blow-up charts gives us the maps
    \[
    (u,t)\mapsto (u,ut,ut^2),\qquad (s,v)\mapsto (s^2v,sv,v).
    \]
    These maps are isomorphisms away from the origin. Over the origin the fiber is $\{u=0\}$ glued to $\{v=0\}$ by $s=t^{-1}.$ Thus the exceptional curve is $\P^1\cong\P(T_pY).$ Since $-\mathrm{id}$ acts trivially on $\P(T_pY),$ the image in the quotient too is $\P^1.$  Moreover, the involution acts by $-1$ in the normal direction. Thus
\[
\operatorname{Tot}\bigl(\OO_{\P^1}(-1)\bigr)/\{\pm1\}
\cong
\operatorname{Tot}\bigl(\OO_{\P^1}(-2)\bigr),
\]
so the image of the exceptional curve has self-intersection $-2$. Therefore $\pi_X$ is locally the minimal resolution of an $A_1$ singularity.

For the second part, let $p,p'\in P_j.$ Then $p'=g(p)$ for some $g\in G$ and thus $e_{p'}=g(e_p).$ Thus all the curves $e_p$ for $p\in P_j$ maps to the same curve under $p:\widetilde{Y}\to\widetilde{Y}/G.$ The image is the unique exceptional $(-2)$-curve over the corresponding $A_1$ singularity, $E_j.$ The exceptional curves $E_j$ on $\widetilde{Y}/G$ are $(-2)$-curves; this implies $\widetilde{Y}/G$ is the minimal resolution and uniqueness gives us $\widetilde{Y}/G\cong X.$
\end{proof}
We get the commutative diagram,
\[\begin{tikzcd}
	{\widetilde{Y}} & Y \\
	X & {\overline{Y}}
	\arrow["q", from=1-1, to=1-2]
	\arrow["p"', from=1-1, to=2-1]
	\arrow["{\pi_Y}", from=1-2, to=2-2]
	\arrow["{\pi_X}"', from=2-1, to=2-2]
\end{tikzcd}\]
The quotient $p:\widetilde{Y}\to X$ is a  finite map of degree $|G|$ between smooth projective surfaces, hence flat. We have,
\begin{equation}\label{eq:pullbackE}
p^*E_j=2\sum_{p\in P_j}e_p,\qquad
p_*e_p=E_j\quad (p\in P_j),\qquad
q_*e_p=0.
\end{equation}
By Bridgeland-King-Reid \cite{BKR},
\[
\Phi:=\R q_*p^*: \Db(X)\xlongrightarrow{\sim}\Dbg(Y)
\]
is an equivalence.

Since $p$ is the quotient map, $p^*F$ carries the canonical $G$-linearization and naturally $p^*F\in \Dbg(\widetilde{Y}).$ Thus under $\Phi$,
\[
F\longmapsto \R q_*(p^*F\text{ with canonical $G$-linearization}).
\]
After forgetting the equivariant structure, one obtains $\Forg\,\circ\,\Phi=\R q_*p^*:\Db(X)\to\Db(Y).$ In other words, the underlying object of $\Phi(E)\in\Dbg(Y)$ is $\R q_*p^*E$; the extra datum in $\Phi(E)$ is the canonical $G$-linearization.

\begin{lemma}\label{lem:action-of-phi}
\begin{enumerate}
\item If $x\notin E_j\subset X$, then $\Phi(\mathcal O_x)$ is the equivariant structure sheaf $\pi_Y^{-1}(\pi_X(x))$. If $x\in E_j$, then $\Phi(\mathcal O_x)=\mathcal O_{Z_x}$ is the $G$-cluster supported on $P_j$, of length $2$ at each point.
\item $\Phi(\mathcal O_{E_j}(-1))=\mathcal O_{P_j}\otimes\operatorname{sgn}$ and $\Phi(\mathcal O_{E_j}(-2))=\mathcal O_{P_j}[-1]$. Thus, for $x\in E_j,$ $\Phi$ takes the triangle $\mathcal O_{E_j}(-1)\to\mathcal O_x\to\mathcal O_{E_j}(-2)[1]$ to the following short exact sequence of equivariant sheaves
\[
0\to\mathcal O_{P_j}\otimes\operatorname{sgn}\to\mathcal O_{Z_x}\to\mathcal O_{P_j}\to0 .
\]
\end{enumerate}
\end{lemma}
\begin{proof}
    \begin{enumerate}
        \item Follows from the definition of the universal family of clusters.
        \item Since $p^*E_j=2\sum_{p\in P_j}e_p,$ the sheaf $p^*(\OO_{E_j}(k))$ is supported on the non-reduced divisor $p^{-1}(E_j)=2\sum_{p\in P_j}e_p$. We have the standard short exact sequence,
        \[
        0\to\OO_{e_p}(-e_p)\to \OO_{2e_p}\to \OO_{e_p}\to 0.
        \]
        Since $e_p^2=-1,$ $\OO_{e_p}(-e_p)=\OO_{e_p}(1).$ Moreover $p$ restricted to $e_p$ gives an isomorphism $p|_{e_p}\colon e_p\xrightarrow{\sim}E_j$, which implies $p^*\OO_{E_j}(k)|_{e_p}\simeq\OO_{e_p}(k)$. Set $L_k=(p^*\OO_{E_j}(k))|_{2e_p},$ which is a line bundle on $2e_p,$ thus tensoring the above sequence by $L_k$ keeps exactness preserved. We get
        \[
        0\to \OO_{e_p}(k+1)\to L_k\to \OO_{e_p}(k)\to 0,
        \]
        since $L_k|_{e_p}\cong\OO_{e_p}(k).$ The stabilizer $\Z/2$ fixes $e_p$ pointwise and acts on its normal bundle by $-1.$ Thus it acts by the character $\sgn$ on $\OO_{e_p}(-e_p)$ and trivially on $\OO_{e_p}$. Consequently, the pieces carry the equivariant structures $\OO_{e_p}(k+1)\otimes\sgn$ and $\OO_{e_p}(k)\otimes 1.$ Now set $k=-1,$
        \[
        0\to\OO_{e_p}\otimes\sgn\to (p^*\OO_{E_j}(-1))|_{2e_p}\to\OO_{e_p}(-1)\to 0.
        \]
        Apply $\R q_*.$ Since $q$ contracts $e_p$ to a point $p,$ 
        \[
        \R q_*\OO_{e_p}(d)=\R \Gamma(\P^1,\OO_{\P^1}(d))\otimes\OO_p.
        \]
        Using this, we get $\R q_*(\OO_{e_p}(-1))=0$ and $\R q_* (p^*\OO_{E_j}(-1))|_{2e_p}= \OO_p\otimes\sgn.$ Since $p^{-1}(E_j)=\sqcup_{p\in P_j}2e_p$, the pullback sheaf decomposes into direct sums, $p^*\OO_{E_j}(-1)\cong\oplus_{p\in P_j}(p^*\OO_{E_j}(-1))|_{2e_p}.$ Doing the calculation for every disjoint component gives,
        \[
        \Phi(\OO_{E_j}(-1))=\R q_*p^*\OO_{E_j}(-1)=\bigoplus_{p\in P_j}\OO_p\otimes\sgn=\OO_{P_j}\otimes\sgn.
        \]
        Similarly, we get the other identity.
    \end{enumerate}
\end{proof}

Before moving to the next section, let us prove that we can apply the general theory to product-quotient surfaces.

\begin{lemma}\label{lem:BKR equivalence PQ}
    Let $X\to (C_1\times C_2)/G$ be a product-quotient surface. Then the canonical sheaf of $C_1\times C_2$ is locally trivial as a $G$-sheaf. In particular, 
    \[
    \Db(X)\cong\Dbg(C_1\times C_2).
    \]
\end{lemma}    
\begin{proof}
    The group $G$ acts faithfully on each $C_i$ and diagonally on $C_1\times C_2,$ so $\operatorname{Fix}_{C_1\times C_2}(g)=\operatorname{Fix}_{C_1}(g)\times \operatorname{Fix}_{C_2}(g)$ for any $g\in G.$ Since $\operatorname{Fix}_{C_i}(g)$ is a closed subvariety on a curve, for any $g\neq 1$, it is finite. Thus $\operatorname{Fix}_{C_1\times C_2}(g)$ is finite, for every $g\neq 1$. Suppose $1\neq g\in G_p$ is a pseudo-reflection on $T_p(C_1\times C_2).$ One can choose local coordinates $(u,v)$ centered at $p$ such that $g(u,v)=(u,\lambda v)$ for $\lambda\neq 1.$ Hence $\operatorname{Fix}_{C_1\times C_2}(g)=\{v=0\},$ but this contradicts $\operatorname{Fix}_{C_1\times C_2}(g)$ is finite. This implies $G_p$ does not contain any pseudo-reflection.

    Recall $(C_1\times C_2)/G$ has only canonical singularities. On surfaces this implies Gorenstein. Now, for a small finite subgroup $G_p\subset\GL(T_p(C_1\times C_2))$ (i.e. no pseudo-reflection), $\C^2/G_p$ is Gorenstein if and only if $G_p\subset\SL(T_p(C_1\times C_2)).$ This implies we can apply \cite{BKR}, which says $\GHilb(C_1\times C_2)\to (C_1\times C_2)/G$ is a crepant resolution. On surfaces, crepant resolution is minimal, so by uniqueness, $\GHilb(C_1\times C_2)\cong X$, and again using \cite{BKR} we get the equivalence of derived categories in the statement. 
\end{proof}

\subsection{The Central Charge}\label{subsection: central-charge}
Let $\overline\omega,\overline B\in\NS(\overline Y)_{\mathbb R}$ with $\overline\omega$ ample, and set
\[
\omega=\pi_Y^*\overline\omega,\quad B=\pi_Y^*\overline B\ \text{on }Y,\qquad \omega_X=\pi_X^*\overline\omega,\quad B_X=\pi_X^*\overline B\ \text{on }X .
\]
Since $\pi_Y$ is finite, $\omega$ is ample and $G$-invariant. On the other hand $\omega_X$ is nef and big ($\omega_X\cdot E_j=0$). By commutativity, $q^*\omega=p^*\omega_X$, $q^*B=p^*B_X$. Take $B^\sharp=B_X+\frac14\sum_{j=1}^nE_j.$

\begin{lemma}\label{lem:GRR}
For every $F\in D^b(X)$,
\[
Z_{\omega,B}\bigl(\Phi(F)\bigr)=Z_{\omega,B}\bigl(Rq_*p^*F\bigr)=|G|\Bigl(Z_{\omega_X,B^\sharp}(F)-\frac n{16}\operatorname{rk}F\Bigr).
\]
\end{lemma}
\begin{proof}
    Recall $q:\widetilde{Y}\to Y$ is the blow-up map. Grothendieck-Riemann-Roch gives $\ch(\R q_*M)=q_*(\ch(M)\cdot\td_q),$ where $\td_q=\td(\widetilde{Y})\cdot q^*\td(Y)^{-1}.$ Write $\td_q=1+\td_q^{(1)}+\td_q^{(2)}$. Under the blow-up $K_{\widetilde{Y}}=q^*K_Y+\sum_{p\in P}e_p.$ Thus,
    \begin{equation}\label{td1}
    \td_q^{(1)}= \td_1(\widetilde{Y})-q^*\td_1(Y)=-\frac{1}{2}(K_{\widetilde{Y}}-q^*K_Y)=-\frac{1}{2}\sum_{p\in P}e_p.
    \end{equation}
    Now apply GRR to $\OO_{\widetilde{Y}}.$ Since $\R q_*\OO_{\widetilde{Y}}=\OO_Y,$ we get $q_*\td_q=1,$ and thus $q_*\td_q^{(1)}=0=q_*\td_q^{(2)}$. Since integration is compatible with proper-pushforward, we get,
    \[
    \int_{\widetilde{Y}}\td_q^{(2)}=\int_Y q_*\td_q^{(2)}=0.
    \]
    Set $\alpha=e^{-B_X-i\omega_X}\ch(F)=\alpha_0+\alpha_1+\alpha_2.$ Thus,
    \[
    \alpha_0=\rk F,\qquad \alpha_1=c_1(F)-\rk F(B_X+i\omega_X),\qquad \int_X\alpha=\int_X\alpha_2=-Z_{\omega_X,B_X}(F).
    \]
    Since the classes on $X$ and $Y$ we defined, came from $\overline{Y},$ we get $q^*(B+i\omega)=p^*(B_X+i\omega_X)$ by the commutativity of the diagram. For any $\omega,B\in\NS(Y)_\R,$
    \begin{align*}
        -Z_{\omega,B}(\R q_*p^*F)&=\int_{Y} e^{-B-i\omega}\ch(\R q_*p^*F)\\
        &=\int_{Y} e^{-B-i\omega} q_*(p^*\ch(F)\cdot\td_q)\\
        &=\int_{Y} q_*(q^*e^{-B-i\omega}p^*\ch(F)\td_q)\qquad\text{projection formula}\\
        &=\int_{\widetilde{Y}} q^*e^{-B-i\omega}p^*\ch(F)\td_q\qquad\text{proper-pushforward}.
    \end{align*}
    Here $q^*e^{-B-i\omega}=e^{q^*(-B-i\omega)}=e^{p^*(-B_X-i\omega_X)}=p^*e^{-B_X-i\omega_X}.$ This gives,
\begin{align*}
    q^*e^{-B-i\omega}p^*\ch(F)&=p^*e^{-B_X-i\omega_X}p^*\ch(F)\\
    &=p^*(e^{-B_X-i\omega_X}\ch(F))\\
    &=p^*\alpha.\\
    \implies -Z_{\omega,B}(\R q_*p^*F)&=\int_{\widetilde{Y}}p^*\alpha\td_q.
\end{align*}
Now only the top-degree part contributes to the integral, and for $p^*\alpha\td_q$ the top-degree part is $p^*\alpha_2+p^*\alpha_1\cdot\td_q^{(1)}+\alpha_0\td_q^{(2)}.$ Notice,
\[
\int_{\widetilde{Y}}\alpha_0\td_q^{(2)}=\rk F\int_{\widetilde{Y}}\td_q^{(2)}=0,\qquad \int_{\widetilde{Y}}p^*\alpha_2=|G|\int_X\alpha_2.
\]
For the middle part we use equation \ref{td1},
\begin{align*}
    \int_{\widetilde{Y}} p^*\alpha_1\cdot\td_q^{(1)}&= -\frac{1}{2}\sum_{p\in P}\int_{\widetilde{Y}}p^*\alpha_1\cdot e_p=-\frac{1}{2}\sum_{p\in P}\alpha_1\cdot p_*e_p.
\end{align*}
For $p\in P_j,$ $p_*e_p=E_j$ and $|P_j|=\frac{|G|}{2}.$ So,
\begin{align*}
    -\frac{1}{2}\sum_{p\in P}\alpha_1\cdot p_*e_p=-\frac{|G|}{4}\sum_j\alpha_1\cdot E_j&=-\frac{|G|}{4}\sum_j(c_1(F)-\rk F(B_X+i\omega_X))\cdot E_j\\
    &=-\frac{|G|}{4}\sum_j c_1(F)\cdot E_j.
\end{align*}
Collecting terms,
\begin{align*}
    -Z_{\omega,B}(\R q_*p^*F)&=|G|\left(\int_X\alpha_2-\frac{1}{4} \sum_j c_1(F)\cdot E_j\right)\\
    &=-|G|\left(Z_{\omega_X,B_X}(F)+\frac{1}{4} \sum_j c_1(F)\cdot E_j\right),
    \end{align*}
since $\int_X\alpha_2=-Z_{\omega_X,B_X}(F).$ Now changing the $B$-field from $B_X$ to $B^\sharp=B_X+\tfrac{1}{4}\sum_j E_j$ gives us,
\[
Z_{\omega_X,B^\sharp}(F)=Z_{\omega_X,B_X}(F)+\frac{1}{4}\sum_j c_1(F)\cdot E_j+\frac{n}{16}\rk F.
\]
The result follows.
\end{proof}

\begin{definition}
    Notice the stability condition $\sigma_{\omega,B}\in\Stab(Y)$ is $G$-invariant by the argument of Proposition \ref{prop:free}. Consider the induced stability condition $(\sigma_{\omega,B})_G$ on $\Dbg(Y).$ Under the equivalence $\Phi,$ define,
\begin{align}\label{main-definition-sigmaX}
    \sigma_X:=|G|^{-1}\cdot\Phi^*(\sigma_{\omega,B})_G\in\operatorname{Stab}(X).
\end{align}
This gives us,
\begin{align*}
    Z_X(F)&=Z_{\sigma_X}(F)=|G|^{-1}\cdot Z_{(\sigma_{\omega,B})_G}(\Phi(F))\\
    &=|G|^{-1}\cdot Z_{\sigma_{\omega,B}}(\Forg\,\circ\,\Phi(F))\\
    &=|G|^{-1}\cdot Z_{\sigma_{\omega,B}}(\R q_*p^*F)\\
    &=Z_{\omega_X,B^{\sharp}}-\frac{n}{16}\rk F.
\end{align*}
Similarly, the slicing $\mathcal P_X(\phi)=\{F: \R q_*p^*F\in\mathcal P_{\omega,B}(\phi)\}$. By Corollary \ref{cor:full-support}, we see that $\sigma_X$ satisfies support property with respect to $\Knum(X);$ thus $\sigma_X\in\Stab(X).$
\end{definition} 

\begin{theorem}\label{maintheorem}
    Every skyscraper sheaf $\mathcal O_x$ is $\sigma_X$-semistable of phase $1$. If $x\notin\bigcup_j E_j$, then $\OO_x$ is $\sigma_X$-stable. If $x\in E_j$, it is strictly semistable with Jordan--H\"older filtration
\[
0\longrightarrow\OO_{E_j}(-1)\longrightarrow\OO_x\longrightarrow\OO_{E_j}(-2)[1]\longrightarrow 0.
\]
Hence, $\sigma_X$ is not geometric. 
\end{theorem}

\begin{proof}
    If $x\notin\cup_j E_j,$ then 
    \[
    \Phi(\OO_x)= \bigoplus_{p\in\pi_Y^{-1}(\pi_X(x))}\OO_p\in\PP_{\omega,B,G}(1),
    \]
    since $\sigma_{\omega,B}$ is geometric. As in the proof of Proposition \ref{prop:free}, one shows that $\Phi(\OO_x)$ is a simple object in $\PP_{\omega,B,G}(1)$ and hence $(\sigma_{\omega,B})_G$-stable. Under the equivalence $\OO_x$ is $\Phi^*(\sigma_{\omega,B})_G$-stable. Since $\C$-action doesn't change stability, $\OO_x$ is $\sigma_X$-stable.

    Now, suppose $x\in E_j$ for some $j.$ By Lemma \ref{lem:action-of-phi}, we know that $\Phi$ takes the triangle $\OO_{E_j}(-1)\to\OO_x\to\OO_{E_j}(-2)[1]$ to the short exact sequence of $G$-equivariant sheaves
    \[
    0\to\OO_{P_j}\otimes\sgn\to \OO_{Z_x}\to\OO_{P_j}\to 0,
    \]
    hence, applying $\Forg$,
    \[
    0\to\OO_{P_j}\to\OO_{Z_x} \to\OO_{P_j}\to 0.
    \]
    Notice, all of these are $0$-dimensional sheaves on $Y$, and thus lie in $\PP_{\omega,B}(1).$ Thus we have a triangle in $\PP_X(1),$ $\OO_{E_j}(-1)\to\OO_x\to\OO_{E_j}(-2)[1].$ In other words, we have the short exact sequence,
    \[
    0\to\OO_{E_j}(-1)\to\OO_x\to\OO_{E_j}(-2)[1]\to 0. 
    \]
    Since $\Forg\,\circ\,\Phi(\OO_{E_j}(-1))=\OO_{P_j}=\bigoplus_{p\in P_j}\OO_p=\Forg\,\circ\,\Phi(\OO_{E_j}(-2)[1]),$ as before, we conclude that both the subobject and quotient in the above short exact sequence are $\sigma_X$-stable. Finally, $\ch(\OO_{E_j}(-1))=(0,E_j,0)$ and $\ch(\OO_{E_j}(-2))=(0,E_j,-1)$ and, since $\omega_X\cdot E_j=0$ and $B^\sharp\cdot E_j=-\tfrac{1}{2}$, we get
\[
Z_X(\OO_{E_j}(-1))=-\frac{1}{2},\qquad Z_X(\OO_{E_j}(-2)[1])=-\bigl(1+B^\sharp\cdot E_j\bigr)=-\frac{1}{2}.
\]
Therefore the short exact sequence is the Jordan--H\"older filtration of $\OO_x:$ its two factors are stable, non-isomorphic, and have the same phase. Hence $\OO_x$ is strictly $\sigma_X$-semistable and $\sigma_X$ is not geometric.
\end{proof}

\subsection*{Mass-Hom bound, deformation, path of stability conditions}
The goal of this section is to prove the following theorem. 

\begin{theorem}\label{thm:maintheorem2}
    Suppose we are in the set up of Theorem \ref{maintheorem}. Then,
\begin{enumerate}
    \item $\sigma_X$ lies in the closure of the geometric chamber $\Stabgeo(X)$.
\item $\sigma_X$ lies in the closure of the locus
\[
\Ubg=\{\,\sigma_{\omega',B'}\cdot g : \omega'\text{ ample},\ B'\in\mathrm{NS}(X)_{\mathbb R},\ g\in\widetilde{\mathrm{GL}}^+(2,\mathbb R)\,\}
\]
if and only if $\omega_X^2\geq n/8.$ In particular, $\mathcal{U}_{\mathrm{BG}} \subsetneq \Stabgeo(X)$ whenever $\omega_X^2<n/8$ for some choice of $\omega.$
\end{enumerate}
\end{theorem}

If $(Z,\PP)$ is a pre-stability condition and $E\in\PP(\phi),$ then $Z(E)=m(E)e^{i\pi\phi}$ with $m(E)\in\R_{>0}.$ This is called the \emph{mass} of the semistable object $E.$ For an object $E\in\Db(X)$ one defines the mass of $E,$ to be $m(E)=\sum_im(E_i),$ where the sum is over the semistable factors of $E$ from the HN filtration of $E.$
\begin{definition}
    \cite[Def. 2.7]{HL-Robotis} We say that a numerical pre-stability condition $\sigma=(Z,\PP)$ admits a \emph{mass-Hom bound} if, for every $A\in\Db(X)$, there exists a constant $C_A\geq 0$ such that
    \[
    \dim_k\Hom(A,F)\leq C_Am_{\sigma}(F)
    \]
    for every $F\in\Db(X)$, where $m_{\sigma}(F)$ denoted the mass of $F$ with respect to $\sigma.$
\end{definition}

\begin{theorem}\label{Cheng}
    \cite[Theorem 2.1]{Cheng:26} Let $X$ be a projective scheme over $\C$, and let $\sigma=(Z,\PP)$ be a numerical
pre-stability condition on $\Db(X)$. Suppose that $\sigma$ admits a mass-Hom bound. Then $\sigma$ satisfies
the support property with respect to the full numerical Grothendieck group $\Knum(X)$.
\end{theorem}

\begin{lemma}\label{lem:sigmaX-mass-Hom}
    $\sigma_X$ as in \ref{main-definition-sigmaX} admits a mass-Hom bound. 
\end{lemma}
\begin{proof}
Recall $\sigma_X=|G|^{-1}\cdot \Phi^*(\sigma_{\omega,B})_G$. By \cite[Prop. 7.6]{LLLMPSZ:26}, we know that every geometric stability condition on a smooth projective surface admits a mass-Hom bound. This implies $\sigma_{\omega,B}\in\Stab(Y)$ (see \S  \ref{subsection: central-charge}) admits a mass-Hom bound. Since $\Phi$ is an equivalence,
\begin{align*}
    \dim\Hom_X(M,N)&=\dim\Hom_{\Dbg (Y)}(\Phi(M),\Phi(N))\\
    &=\dim\Hom_{Y}(\Forg\,\circ\,\Phi(M),\Forg\,\circ\,\Phi(N))^G\\
    &=\dim\Hom_{Y}(\R q_*p^*(M),\R q_*p^*(N))^G\\
    &\leq \dim\Hom_{Y}(\R q_*p^*(M),\R q_*p^*(N)).
\end{align*}
Let $\R q_*p^*(M)=M_Y,\,\R q_*p^*(N)=N_Y.$ Since $\sigma_{\omega,B}$ admits a mass-Hom bound, there exists a constant $C_{M_Y}\geq 0$, such that
\[
\dim\Hom_X(M,N)\leq C_{M_Y}m_{\sigma_{\omega,B}}(N_Y)
\]
Let $0\neq F\in\Db(X)$ with $\sigma_X$-Harder--Narasimhan filtration $0=F_0\to\dots \to F_m=F$ and semistable factors $A_i\in\PP_X(\phi_i),$ $\phi_1>\dots>\phi_m.$ By definition of the slicing $\PP_X$ we have $\Forg\,\circ\,\Phi(A_i)\in\PP_{\omega,B}(\phi_i),$ hence $0=F_{0,Y}\to\dots \to F_{m,Y}=F_{Y}$ is the Harder-Narasimhan filtration of $F_Y,$ where $F_{i,Y}=\R q_*p^*(F_i).$ Since $Z_X=|G|^{-1}Z_{\omega,B}\circ\R q_*p^*,$ 
\[
m_{\sigma_X}(F)=\sum_i|Z_X(A_i)|=|G|^{-1}\sum_i|Z_{\omega,B}(A_{i,Y})|=|G|^{-1}m_{\sigma_{\omega,B}}(F_Y).
\]
Thus,
\[
m_{\sigma_{\omega,B}}(N_Y)=|G|\,m_{\sigma_X}(N).
\]
Therefore,
\[
\dim \Hom_X(M,N)\leq |G|\,C_{M_Y}\,m_{\sigma_X}(N).
\]
Hence, $\sigma_X$ admits a mass-Hom bound.
\end{proof}

\begin{corollary}\label{cor:full-support}
    $\sigma_X$ satisfies the support property with respect to the full numerical Grothendieck group $\Knum(X)$.
\end{corollary}
\begin{proof}
    Follows from Lemma \ref{lem:sigmaX-mass-Hom} and Theorem \ref{Cheng}.
\end{proof}

\begin{lemma}\label{lem:closure of UBG}
    When $\omega_X^2<n/8,$ $\sigma_X$ does not lie in the closure of $\Ubg=\{\,\sigma_{\omega',B'}\cdot g : \omega',B'\in\NS(X)_\R\text{ with $\omega'$ ample},\ g\in\widetilde{\mathrm{GL}}^+(2,\mathbb R)\,\}.$
\end{lemma}

\begin{proof}
    For a surface $X,$ we know $\Knum(X)_\R\cong \R\oplus\NS(X)_\R\oplus\R$ with the Mukai pairing $\langle(r,c,s),(r',c',s')\rangle=c\cdot c'-rs'-r's.$ Since the pairing is nondegenerate, for every $Z\in\Hom(\Knum(X)_\R,\C)$ there exists a unique $\Omega_Z\in\Knum(X)_\R,$ such that,
    \[
    Z(-)=\langle\Omega_Z,-\rangle.
    \]
    Now consider the following function in $Z,$
    \[
    Q(Z)=\langle\Re\Omega_Z,\Re \Omega_Z \rangle\langle \Im\Omega_Z, \Im\Omega_Z\rangle-\langle\Re\Omega_Z,\Im\Omega_Z \rangle^2.
    \]
    This is continuous on $\Hom(\Knum(X)_\R,\C).$ Indeed, $Z\mapsto\Omega_Z$ and $\Omega\mapsto (\Re\Omega,\Im\Omega)$ are $\R$-linear, and the Mukai pairing is bilinear, so $Q$ is a homogeneous polynomial, thus continuous.
    
    For $Z_{\omega,B}$ we have $\Omega=e^{(B+i\omega)}= \bigl(1,B+i\omega,\tfrac{B^2-\omega^2}{2}+iB\cdot\omega\bigr).$ From this, one computes, $Q(Z_{\omega,B})=\omega^4>0,$ since $\omega$ is ample. For any $g\in\widetilde{\GL}_2^+(\R)$, we have $Q(g\circ Z)=\det(g)^2Q(Z).$ Hence $Q>0$ on $\Ubg$ and by continuity of $\Stab(X)\to \Hom(\Knum(X),\C),$ $Q\geq 0$ on $\overline{\UU_{BG}},$

    Now back to our case, $Z_X=Z_{\omega_X,B^\sharp}-\tfrac{n}{16}\rk$. One has $\Omega_{Z_X}=e^{B^\sharp+i\omega_X}+(0,0,\tfrac{n}{16}).$ This gives,
    \begin{align*}
        Q(Z_X) &=\langle\Re\Omega_{Z_X},\Re \Omega_{Z_X} \rangle \langle \Im\Omega_{Z_X}, \Im\Omega_{Z_X}\rangle-\langle\Re\Omega_{Z_X},\Im\Omega_{Z_X} \rangle=\left(\omega_X^2-\frac{n}{8}\right)\omega_X^2.
    \end{align*}

Since $\omega_X$ is nef and big, $\omega_X^2>0.$ Therefore whenever $\omega_X^2<n/8,$ $Q(Z_X)<0$, which implies $\sigma_X\notin\overline{\Ubg}.$
\end{proof}

\begin{lemma}\label{lem:mainlemma-path}
    Let $\sigma_X\in\Stab(X)$ and $H$ be an ample class on $X.$ Set
    \[
    W_t=Z_X+it(H\cdot c_1 - (H\cdot B^\sharp)\rk)+t\rk,
    \]
    for $t\in[0,t_1).$ There exist a $t_0\in (0,t_1]$ and a continuous path $(\tau_t)_{0\leq t<t_0}$ in $\Stab(X)$ with $\tau_0=\sigma_X$, central charge $W_t.$ Moreover, $\tau_t\in\Stabgeo(X)$ for $t\in(0,t_0).$
\end{lemma}
\begin{proof}
    The path $t\mapsto W_t$ is continuous in $\Hom(\Knum(X),\C)$ with $W_0=Z_X$. Since $\Stab(X)\to\Hom(\Knum(X),\C)$ is a local homeomorphism near $\sigma_X$, we get a $t_1>0$ and a continuous path $(\tau_t)_{0\leq t<t_1}$ in $\Stab(X)$ with $\tau_0=\sigma_X$ and central charge of $\tau_t$ is $W_t.$ Set $\tau_t=(W_t,\PP_t).$ We may assume (after shrinking) that $d(\sigma_X,\tau_t)<\tfrac{1}{8}$ for every $t<t_1,$ where $d(-,-)$ is the generalized metric on $\Stab(X)$ \cite[Prop. 8.1]{Bri-original}. This implies, for every $0\neq F$ and $t<t_1,$ we have
    \begin{equation}\label{metric-inequalities}
        |\phi^+_{\sigma_X}(F)-\phi^+_{\tau_t}(F)|<\tfrac{1}{8},\quad |\phi^-_{\sigma_X}(F)-\phi^-_{\tau_t}(F)|<\tfrac{1}{8},\quad e^{-1/8}<\tfrac{m_{\sigma_X}(F)}{m_{\tau_t}(F)}<e^{1/8}.
    \end{equation}

    Since $\ch(\OO_x)=(0,0,1),\ch(\OO_{E_j}(-1))=(0,E_j,0),$ and $\ch(\OO_{E_j}(-2)[1])=(0,-E_j,1),$ we have,
    \begin{equation}\label{central-charge-W_t}
        W_t(\OO_x)=-1,\qquad
  W_t(\OO_{E_j}(-1))=-\tfrac{1}{2}+it(H\cdot E_j),\qquad
  W_t(\OO_{E_j}(-2)[1])=-\tfrac{1}{2}-it(H\cdot E_j),
    \end{equation}
    where $H\cdot E_j>0$ since $H$ is ample. 
    
    Since $\OO_x$ is $\sigma_X$--semistable with phase $1$ (Theorem \ref{maintheorem}), $\phi_{\sigma_X}^+(\OO_x)=1= \phi_{\sigma_X}^-(\OO_x),$ and $m_{\sigma_X}(\OO_x)=|Z_{X}(\OO_x)|=1.$ By \ref{metric-inequalities},
    \[
    \frac{7}{8}<\phi_{\tau_t}^-(\OO_x)\leq\phi_{\tau_t}^+(\OO_x)<\frac{9}{8},\qquad m_{\tau_t}(\OO_x)<e^{1/8}<2.
    \]
    Thus every $\tau_t$-HN factor $A$ of $\OO_x$ is $\tau_t$--semistable of some phase $\psi\in \left(\tfrac{7}{8},\tfrac{9}{8}\right)$. By \ref{metric-inequalities}, $\phi^{\pm}_{\sigma_X}(A)\in (\psi-\tfrac{1}{8},\psi+\tfrac{1}{8})\subset (\tfrac{3}{4},\tfrac{5}{4}),$ i.e. $A\in \PP_X((\tfrac{3}{4},\tfrac{5}{4})).$ In particular, $\OO_x$ and its $\tau_t$--HN factors all lie in the heart $\AA:=\PP_X((\tfrac{1}{2},\tfrac{3}{2}])$. Notice, $\Im Z_X(E)=\Im Z_X(\OO_x)-\Im Z_X(M)=0$ by 

    Let $\Gamma\subset\Knum(X)$ be the set of all classes $[F]$ of objects $F\neq 0$ with $m_{\tau_t}(F)\leq 2$ for some $t<t_1.$ Support property implies $\Gamma$ is finite. Since $\Gamma$ is finite and $W_t\to Z_X$ as $t\to 0,$ there exists a $t_0\in(0,t_1]$ such that,
    \begin{equation}\label{def:t_0}
        \operatorname{sign} \Im W_t(\gamma)=\operatorname{sign} \Im Z_X(\gamma),\quad \text{for all $t\in(0,t_0)$ and all $\gamma\in\Gamma$ with $\Im Z_X(\gamma)\neq 0.$}
    \end{equation}
    Fix $t\in (0,t_0)$ and $x\in X.$ We claim that there's no short exact sequence as follows,
    \begin{equation}\label{destab-seq}
        0\to M\to\OO_x\to N\to 0
    \end{equation}
    with $M,N\neq 0,$ $m_{\tau_t}(M)<2$ and $\Im W_t(M)\leq 0.$ 
    
    Since $M\in\AA$, the $\sigma_X$--HN factors on $M$ also lie in $\AA$ Notice $\phi_{\sigma_X}^+(M)\leq 1$ and $\Im Z_X(M)\geq 0.$ Let $M_1\hookrightarrow M$ be the first step on the $\sigma_X$--HN filtration, so $M_1\in\PP_X(\phi_{\sigma_X}^+(M)).$ The composite $M_1\hookrightarrow M\hookrightarrow\OO_x$ is a monomorphism in $\AA$, hence nonzero. Since $\OO_x\in\PP_X(1)$, this forces $\phi^+_{\sigma_X}(M)\leq1$. Thus all the $\sigma_X$--HN factors of $M$ have phases in $(\tfrac{1}{2},1]$, i.e. the imaginary part under $Z_X$ is non-negative and $\Im Z_X(M)\geq 0.$

    Suppose $\Im Z_X(M)>0.$ Since $m_{\tau_t}(M)<2$, we have $[M]\in\Gamma$, thus \ref{def:t_0} implies $\Im W_t(M)>0,$ contradicting $\Im W_t\leq 0.$

    Now take the remaining case; let $\Im Z_X(M)=0.$ Since all the $\sigma_X$--HN factors of $M$ have phases in $(\tfrac{1}{2},1]$ and non-negative imaginary parts under $Z_X$ summing to $0,$ we conclude that they all have phase exactly $1.$ This proves $M\in\PP_X(1).$ We now prove that also $N\in\PP_X(1).$  Notice $\Im Z_X(N)=\Im Z_X(\OO_x)-\Im Z_X(M)=0$ by \ref{central-charge-W_t}. Let $N'\in\PP_X(\sigma_X^-(N))$ be the last $\sigma_X$--HN facor of $N$ and let $N\to N'$ be the projection, which is a quotient in $\AA.$ Since the composition $\OO_x\to N\to N'$ is an epimorphism in $\AA,$ it is non-zero, which implies $\phi_{\sigma_X}^-(N)\geq 1.$ Thus all the $\sigma_X$--HN factors of $N$ have phase in $[1,\tfrac{3}{2}].$ This means they have non--positive imaginary part under $Z_X.$ But since they sum to $\Im Z_X(N)=0,$ every factor has phase $1$ and $N\in \PP_X(1)$. 

    Therefore, \ref{destab-seq} is a short exact sequence in the abelian category in $\PP_X(1)$, since its a triangle with all three terms in $\PP_X(1).$ If $x\notin\cup_j E_j,$ then $\OO_x$ is simple in $\PP_X(1)$ by Theorem \ref{maintheorem}, a contradiction. On the other hand, if $x\in E_j$, then again by Theorem \ref{maintheorem}, $\OO_x$ has a length $2$ Joran--H\"older filtration. Thus, $\{M,N\}\cong\{\OO_{E_j}(-1),\OO_{E_j}(-2)[1]\}.$ Notice, $\Hom(\OO_{E_j}(-2)[1],\OO_x)=0.$ Thus, $M\cong \OO_{E_j}(-1).$ But then, 
    \[ 
    \Im W_t(M)=\Im W_t(\OO_{E_j}(-1))=t(H\cdot E_j)>0,
    \]
    again contradictory to $\Im W_t(M)\leq 0.$

    Now, fix $t\in (0,t_0)$ and $x\in X.$ Suppose $\OO_x$ is not $\tau_t$--semistable. Let $M\hookrightarrow\OO_x$ be the first step of the $\tau_t$--HN filtration and let $N$ be the cone of this map. Then as we saw above, $M$ and the remaining $\tau_t$--HN factors lie in $\PP_X((\tfrac{3}{4},\tfrac{5}{4}))\subset\AA.$ Since $N$ is an iterated extension of the remaining factors and $\AA$ is closed under extension, we get $N\in\AA.$ Therefore, the triangle $M\to\OO_x\to N$ has all terms in $\AA$ and is a short exact sequence in $\AA$ as \ref{destab-seq} with $M,N\neq 0.$ Moreover, $m_{\tau_t}(M)\leq m_{\tau_t}(\OO_x)<2$ and the $\tau_t$-phases of the HN factors lie in $(\tfrac{7}{8},\tfrac{9}{8})$ and are strictly decreasing. Notice, if the first phase were $\leq 1$, then all phases would lie in $(\tfrac{7}{8},1]$, so all factors would have non-negative imaginary parts under $W_t$. Since they all sum to $\Im W_t(\OO_x)=0$, all factors would have phase $1$, contradicting strict decrease in phases. Hence $\phi_{\tau_t}(M)\in(1,\tfrac98)$ and $\Im W_t(M)<0$, again a contradiction to the above discussion.

    Thus, $\OO_x$ is $\tau_t$--semistable with $W_t(\OO_x)=-1$, so $\phi_t(\OO_x)\in 1+2\Z$. Since $\phi_t(\OO_x)\in(\tfrac{7}{8},\tfrac{9}{8}),$ we get $\phi_t(\OO_x)=1,$ for every $x\in X.$ Suppose $\OO_x$ is strictly $\tau_t$--semistable. Then, by the support property, $\tau_t$ is locally finite. Thus $\PP_t(1)$ has finite length and we can choose a $\tau_t$--stable subobject $M\subsetneq \OO_x$ in $\PP_t(1)$ with quotient $0\neq N\in \PP_t(1).$ Then $W_t(M),W_t(N)\in\R_{<0}$, so $\Im W_t(M)=0$ and $m_{\tau_t}(M)=|W_{t}(M)|<|W_t(\OO_x)|=1<2.$ By \ref{metric-inequalities}, $M,N\in\PP_X((\tfrac{7}{8},\tfrac{9}{8}))\subset\AA$. Thus the triangle $M\to\OO_x\to N$ is a short exact sequence \ref{destab-seq} in $\AA,$ again contradicting the claim above. Thus every $\OO_x$ is $\tau_t$--stable of phase $1.$ This proves $\tau_t\in\Stabgeo(X)$ for all $t\in (0,t_0).$
    \end{proof}

\subsubsection*{Proof of Theorem \ref{thm:maintheorem2}}
We can now prove the theorem we promised at the start of this section.
\begin{enumerate}
    \item Using Lemma \ref{lem:mainlemma-path}, we have $\tau_t\in\Stabgeo(X)$ for $t\in(0,t_0)$ and $\tau_t\to\sigma_X$ as $t\to 0.$ Thus $\sigma_X\in\overline{ \Stabgeo(X)}.$
    
    \item From Lemma \ref{lem:closure of UBG} it follows that $\sigma_X\in\overline{\Ubg}\implies\omega_X^2\geq n/8.$ Conversely, assume $\omega_X^2\geq n/8$ and let $(\tau_t)$ be the path of Lemma \ref{lem:mainlemma-path}. Let $H_t=\omega_X+tH$ for $t\in(0,t_0),$ which is ample since $\omega_X$ is nef and $H$ is ample. Let $g_b\in\GL_2^+(\R)$ acts as $g_b(u+iv)=u+ibv.$ Now define,
    \[
    b_t=\left(\frac{\omega_X^2-n/8+2t}{H_t^2}\right)^{1/2}>0.
    \]
    Comparing the real and imaginary parts,
    \[
    W_t=g_{1/b_t}\,\circ\, Z_{b_tH_t,B^\sharp}.
    \]
    Let $\widetilde{g}_t \in\widetilde{\GL}_2^+ (\R)$ be the lift of $g_{1/b_t}.$ The stability condition $\tau_t\cdot \widetilde{g}_t$ has central charge $Z_{b_tH_t,B^\sharp}$ with $b_tH_t$ ample, has the same semistable objects as $\tau_t$, so every skyscraper sheaf is $\tau_t\cdot\widetilde{g}_t$-stable of phase $1$. We get $\tau_t\cdot\widetilde{g}_t =\sigma_{b_tH_t,B^\sharp}$, i.e.
    \[
    \tau_t= \sigma_{b_tH_t,B^\sharp} \cdot \widetilde{g}_t^{-1}\in \Ubg,\qquad\text{ for }0<t<t_0,
    \]
    by Prop. \ref{geometric-stab}. Thus $\sigma_X\in\overline{\Ubg}.$
\end{enumerate}

\qed

\begin{remark}[Comparison with \cite{Chou,Tramel-Xia},]
The stability conditions $\sigma_X$ should be compared with the ones constructed by Tramel--Xia \cite{Tramel-Xia} and Chou \cite{Chou}. The two constructions are not exactly the same. For $x\in E_j$, the Jordan--H\"older factors of $\OO_x$ are $\OO_{E_j}(-1)$ and $\OO_{E_j}(-2)[1]$ for $\sigma_X$ (Theorem \ref{maintheorem}), but $\mathcal{O}_{E_j}$ and $\mathcal{O}_{E_j}(-1)[1]$ in  \cite[Proof of Lemma 8.9]{Chou}. Notice that the spherical twist along $\mathcal{O}_{E_j}(-1)$ sends the second pair to the first, and acts on $K_{\mathrm{num}}(X)$, which turns the class
$B^\sharp=B_X+\frac{1}{4}\sum_j E_j$ in the definition of $Z_X$ to $B_X-\frac{1}{4}\sum_j E_j$, which looks like the classes appearing in the central charges of \cite[\S4]{Chou}. This suggests that $\sigma_X$ should be related to the stability conditions of \cite{Chou} by these spherical twists. We plan to return to this question soon in future work.
\end{remark}

\bibliographystyle{alpha}
\bibliography{mainbib}

\end{document}